\documentclass{article}
\usepackage{amsmath,amsthm,amssymb,graphics,epic}
\usepackage[latin1]{inputenc}

\newtheorem{theorem}{Theorem}
\newtheorem{corollary}[theorem]{Corollary}

\begin{document}

\title{New partition function recurrences}

\author{Robson da Silva  \\
	Universidade Federal de S\~ao Paulo \\
	S\~ao Jos\'e dos Campos - SP, 12247-014, Brazil \\ silva.robson@unifesp.br \vspace{0.25cm} \\
Pedro Diniz Sakai \\
Universidade Federal de S\~ao Paulo \\
S\~ao Jos\'e dos Campos - SP, 12247-014, Brazil \\ pedro.sakai@unifesp.br 
}

\date{}
\maketitle

\begin{abstract}
We present Euler-type recurrence relations for some partition functions. Some of our results present new recurrences for the number of unrestricted partitions of $n$, denote by $p(n)$. Others establish recurrences for partition functions not yet considered. 
\end{abstract}

\noindent {\bf keywords}: Partition, overpartition, recurrence relation, generating function


\section{Introduction} 
\label{sect1}

A partition of an integer $n$ is a finite set of positive integers $\{ \lambda_{1}, \ldots, \lambda_{s}\}$ such that $n = \lambda_{1} + \cdots + \lambda_{s}$. The $\lambda_{i}$s are called the parts of the partition. The number of partitions of $n$ is usually denote by $p(n)$, with $p(0)=1$ by convention. For example, we have $p(4) = 5$ since there are five partitions of 4, namely:
$$4, 3 + 1, 2 + 2, 2 + 1 + 1, 1 + 1 + 1 + 1.$$
The generating function of $p(n)$, due to Euler, is given by (see \cite[eq. (1.1.6)]{Berndt})
\begin{equation}
\sum_{n=0}^{\infty} p(n)q^n = \prod_{k=1}^{\infty} \frac{1}{1-q^k}.
\label{GF}
\end{equation}
So, one can obtain the values of $p(n)$ by expanding the right-hand side of \eqref{GF} and extracting the coefficient of $q^n$. Another way to obtain $p(n)$ was found by Euler after he proved the following identity (known as Euler's pentagonal number theorem):
\begin{equation}
\sum_{n=-\infty}^{\infty} (-1)^n q^{n(3n-1)/2} = \prod_{k=1}^{\infty} (1-q^{k}).
\label{PNT} 
\end{equation}
Indeed, multiplying \eqref{GF} and \eqref{PNT} we obtain
\begin{equation*}
\sum_{n=0}^{\infty} p(n)q^n \sum_{n=-\infty}^{\infty} (-1)^n q^{n(3n-1)/2} = 1,
\end{equation*}
from which the following recurrence for $p(n)$ is derived after extracting the coefficient of $q^n$ from its both sides:
\begin{align*}
p(n)-p(n-1)-p(n-2)+p(n-5)+p(n-7)-p(n-12)-p(n-15) \\
 + \cdots + (-1)^jp(n-j(3j-1)/2) + (-1)^jp(n-j(3j+1)/2) + \cdots \\
= \begin{cases}
1, & \mbox{if $n=0$,} \\
0, & \mbox{otherwise.}
\end{cases}
\end{align*}
The numbers $j(3j\pm 1)/2$ are the pentagonal numbers.

Some subsequent works bring new recurrence relations for $p(n)$ and other partition functions. Ewell \cite[Theorem 2]{Ewell}, for instance, presents the following recurrence for $p(n)$ involving the triangular numbers 
\begin{align*}
p(n)-p(n-1)-p(n-3)+p(n-6)+p(n-10)-p(n-15)-p(n-21) \\
+ \cdots + (-1)^jp(n-j(2j-1)) + (-1)^jp(n-j(2j+1)) + \cdots \\
= \begin{cases}
0, & \mbox{if $n$ is odd,} \\
q(n/2), & \mbox{if $n$ is even,}
\end{cases}
\end{align*}
where $q(n)$ denotes the number of partitions of $n$ into distinct parts. Merca \cite{Merca} derived two new recurrence relations for $p(n)$, which allowed him to obtain a more efficient method to compute the parity of $p(n)$. Ono, Robbins, and Wilson \cite{Ono} presented recurrence relations for some partition functions, including $q(n)$, $q_O(n)$ the number of partitions into distinct and odd parts, $p_E(n)$ the number of partitions into an even number of parts, and $p_O(n)$ the number of partitions into an odd number of parts. Recently, Choliy, Kolitsch, and Sills \cite{Sills} found a number of new recurrences for $p(n)$, including
\begin{align*}
p(n)-p(n-1)-p(n-2)+p(n-4)+p(n-8)-p(n-9)-p(n-18) \\
+ \cdots + (-1)^jp(n-j^2) + (-1)^jp(n-2j^2) + \cdots \\
= \begin{cases}
0, & \mbox{if $n$ is odd,} \\
qq(n), & \mbox{if $n$ is even,}
\end{cases}
\end{align*}
and
\begin{align*}
p(n)-2p(n-1)+2p(n-4)-2p(n-9)+2p(n-16) + \cdots \\
+ (-1)^j2p(n-j^2) + \cdots = (-1)^n qq(n).
\end{align*}
where $qq(n)$ denotes the number of partitions of $n$ into distinct odd parts. Additional recurrence relations for partition functions can be found in \cite{Sills,Ewell,Ewell1,Merca,N,Ono}. 

In this paper, using some classical identities and generating function manipulations, we provide a number of new recurrence relations for $p(n)$, $qq(n)$, $\overline{p}(n)$ the number of overpartitions of $n$, $p_o(n)$ the number of partitions of $n$ into odd parts, and $p^c_m(n)$ the number of partitions of $n$ into parts congruent to $\pm c$ modulo $m$. {For some of these functions, it is the first time that recurrence relations are presented.}

\section{Preliminaries}
\label{Sect_prelim}

We recall Ramanujan's theta functions
\begin{align}
f(a,b) & :=\sum_{n=-\infty}^\infty a^\frac{n(n+1)}{2}b^\frac{n(n-1)}{2}, \mbox{ for } |ab|<1, \label{eq1} \\
\psi(q) & := f(q,q^3) = \sum_{n=0}^{\infty} q^{n(n+1)/2}. \label{eq3}
\end{align}

In the proofs of some of our results, we will need Jacobi triple product identity given by (see \cite[Theorem 1.3.3]{Berndt})
\begin{align*}
\sum_{n=-\infty}^{\infty} z^nq^{n^2}=(-zq;q^2)_\infty(-q/z;q^2)_\infty(q^2;q^2)_\infty.
\end{align*}
where we use the following standard $q$-series notation:
\begin{align*}
(a;q)_0 & = 1, \\
(a;q)_n & = (1-a)(1-aq) \cdots (1-aq^{n-1}), \forall n \geq 1, \\ 
(a;q)_{\infty} & = \lim_{n \to \infty} (a;q)_n, |q|<1.
\end{align*}	
Using \eqref{eq1}, we can rewrite Jacobi triple product identity in the form
\begin{align}
f(a,b) = (-a;ab)_\infty(-b;ab)_\infty(ab;ab)_\infty.
\label{eq5}
\end{align}
	
An important consequence of \eqref{eq5} is following identity (see \cite[Eq. (1.3.14)]{Berndt})
\begin{align}
\psi(q) = \frac{(q^2;q^2)_\infty}{(q;q^2)_\infty}.
\label{eq7}
\end{align}

We also recall the well-known Euler's pentagonal number theorem (see \cite[Corollary 1.3.5]{Berndt}):
\begin{equation}
(q;q)_\infty = \sum_{n=-\infty}^{\infty} (-1)^n q^{n(3n-1)/2}.
\label{Euler}
\end{equation}

\section{Main results}
\label{Sect_main}

In what follows, we let $t_j^e$ (resp. $t_j^o$) denote the $j$-th even (resp. odd) triangular number. So, $t_1^e = 0$, $t_1^o = 1$, $t_2^e = 3$, $t_2^o = 6$, $t_3^e = 10$, $t_3^o = 15$, etc.

\begin{theorem} For all even integer $n \geq 1$, we have
	\begin{align}
	& p(n/2)+p((n-6)/2)+p((n-10)/2)+p((n-28)/2)+p((n-36)/2)+ \nonumber \\
	&p((n-66)/2)+p((n-78)/2)+\cdots+p((n-t^e_j)/2)+\cdots= p_d(n),
\label{Eq2}
	\end{align}	
where $p_d(n)$ denotes the number of partitions of $n$ into distinct parts. For all odd integer $n \geq 1$, we have
	\begin{align}
	& p((n-1)/2)+p((n-3)/2)+p((n-15)/2)+p((n-21)/2)+ \nonumber \\
	&p((n-45)/2)+p((n-55)/2)+\cdots+p((n-t^o_j)/2)+\cdots= p_o(n),
\label{Eq3}
	\end{align}
where $p_o(n)$ denotes the number of partitions of $n$ into odd parts
	\label{Teo3.2}
\end{theorem}

\begin{proof} Initially, we note that 
	\begin{align*}
	\sum_{n=0}^{\infty}\frac{1+(-1)^n}{2}p_d(n)q^n&=\frac{1}{2}\bigg(\sum_{n=0}^{\infty} p_d(n)q^n+\sum_{n=0}^{\infty} p_d(n)(-q)^n\bigg)\\
	&=\frac{1}{2}\bigg(\prod_{k=1}^{\infty}(1+q^k)+\prod_{k=1}^{\infty} (1+(-1)^kq^k)\bigg) \\
	&=\frac{1}{2}\bigg(\prod_{k=1}^{\infty}(1+q^k)+\prod_{k=1}^{\infty} (1+q^{2k})(1-q^{2k-1})\bigg)
	\end{align*}
and
	\begin{align*}
	\prod_{k=1}^{\infty} (1+q^k)&=\prod_{k=1}^{\infty} (1+q^k)\frac{(1-q^k)}{(1-q^k)}\frac{(1-q^{2k-1})}{(1-q^{2k-1})} \\
	&=\prod_{k=1}^{\infty}\frac{(1-q^{2k-1})(1-q^{2k})}{(1-q^k)(1-q^{2k-1})} =\prod_{k=1}^{\infty}\frac{1}{1-q^{2k}}\psi(q).
	\end{align*}
We also have
	\begin{align*}
	\prod_{k=1}^{\infty} (1+q^{2k})(1-q^{2k-1})&=\prod_{k=1}^{\infty} (1+q^{2k})(1-q^{2k-1})\frac{(1-q^{2k})}{(1-q^k)(1+q^k)}\\
	&=\prod_{k=1}^{\infty} \frac{(1-q^{2k-1})(1-q^{2k})}{(1-q^k)(1+q^{2k-1})} =\prod_{k=1}^{\infty} \frac{1}{1-q^{2k}}\psi(-q).
	\end{align*}
It follows that
	\begin{align*}
	\sum_{n=0}^{\infty}\frac{1+(-1)^n}{2}p_d(n)q^n&=\frac{1}{2}\bigg(\prod_{k\geq1}\frac{1}{1-q^{2k}}\psi(q)+\prod_{k=1}^{\infty} \frac{1}{1-q^{2k}}\psi(-q)\bigg)\\
	&=\bigg(\prod_{k=1}^{\infty} \frac{1}{1-q^{2k}}\bigg)\frac{1}{2}\bigg(\psi(q)+\psi(-q)\bigg)\\
	&=\bigg(\prod_{k=1}^{\infty} \frac{1}{1-q^{2k}}\bigg)\frac{1}{2}\bigg(\sum_{j=0}^{\infty}q^{\frac{j(j+1)}{2}}+\sum_{j=0}^{\infty}(-q)^{\frac{j(j+1)}{2}}\bigg).
	\end{align*}
The parity of the exponent ${\frac{j(j+1)}{2}}$ is given by
	\begin{align}
	&\frac{4i(4i+1)}{2}= 8i^2+2i, 
	\label{eq31}\\
	&\frac{(4i-1)(4i-1+1)}{2}= 8i^2-2i, 
	\label{eq32}\\
	&\frac{(4i-2)(4i-2+1)}{2}= 8i^2-6i+1,
	\label{eq33}\\
	&\frac{(4i-3)(4i-3+1)}{2}= 8i^2-10i+3.
	\label{eq34}
	\end{align}
The even triangular numbers are given by \eqref{eq31} and \eqref{eq32}, while \eqref{eq33} and \eqref{eq34} represent the odd triangular numbers. Thus	
\begin{align}
\sum_{j=0}^{\infty}q^{\frac{j(j+1)}{2}}=\sum_{i=0}^{\infty} q^{8i^2+2i}+q^{8i^2-2i}+q^{8i^2-6i+1}+q^{8i^2-10i+3}
\label{Eq4}
\end{align}
and
\begin{align}
\sum_{j=0}^{\infty}(-q)^{\frac{j(j+1)}{2}}=\sum_{i=0}^{\infty}q^{8i^2+2i}+q^{8i^2-2i}-q^{8i^2-6i+1}-q^{8i^2-10i+3},
\label{Eq5}
\end{align}
which yields
\begin{align*}
\sum_{n=0}^{\infty}\frac{1+(-1)^n}{2}p_d(n)q^n&=\bigg(\prod_{k=1}^{\infty}\frac{1}{1-q^{2k}}\bigg)\frac{1}{2}\bigg(\sum_{i=0}^{\infty}2q^{8i^2+2i}+2q^{8i^2-2i}\bigg)\\
&=\bigg(\sum_{k=0}^{\infty}p(k)q^{2k}\bigg)\bigg(\sum_{j=0}^{\infty}q^{t^e_j}\bigg)\\
&=\sum_{k=0}^{\infty}\bigg(\sum_{j=0}^{\infty}p(k)q^{2k+t^e_j}\bigg).
\end{align*}
Now we extract the coefficient of $q^n$ on both sides of the above equation to obtain	
\begin{align*}
\sum_{n=0}^{\infty}p_d(n)q^n=\sum_{n=0}^{\infty}\bigg(\sum_{j=0}^{\infty}p((n-t^e_j)/2)\bigg)q^n,
\end{align*}
which completes the proof of \eqref{Eq2}.

In order to prove \eqref{Eq3}, we begin with
	\begin{align*}
	\sum_{n=0}^{\infty}\frac{1-(-1)^n}{2}p_o(n)q^n&=\frac{1}{2}\bigg(\sum_{n=0}^{\infty}p_o(n)q^n-\sum_{n=0}^{\infty}p_o(n)(-q)^n\bigg)\\
	&=\frac{1}{2}\bigg(\prod_{k=1}^{\infty}\frac{1}{(1-q^{2k-1})}-\prod_{k=1}^{\infty}\frac{1}{(1-(-q)^{2k-1})}\bigg) \\
	&=\frac{1}{2}\bigg(\prod_{k=1}^{\infty}\frac{1}{(1-q^{2k-1})}-\prod_{k=1}^{\infty}\frac{1}{(1+q^{2k-1})}\bigg).
	\end{align*}
We note that
	\begin{align*}
	\prod_{k=1}^{\infty}\frac{1}{(1-q^{2k-1})}&=\prod_{k=1}^{\infty}\frac{1}{(1-q^{2k-1})}\frac{(1-q^{2k})}{(1-q^{2k})} =\prod_{k=1}^{\infty}\frac{1}{1-q^{2k}}\psi(q)
	\end{align*}
and
	\begin{align*}
	\prod_{k=1}^{\infty}\frac{1}{(1+q^{2k-1})}&=\prod_{k=1}^{\infty}\frac{1}{(1+q^{2k-1})}\frac{(1-q^{2k})}{(1-q^{2k})} =\prod_{k=1}^{\infty}\frac{1}{1-q^{2k}}\psi(-q).
	\end{align*}
It follows that
	\begin{align*}
	\sum_{n=0}^{\infty}\frac{1-(-1)^n}{2}p_o(n)q^n&=\frac{1}{2}\bigg(\prod_{k=1}^{\infty}\frac{1}{1-q^{2k}}\psi(q)-\prod_{k=1}^{\infty}\frac{1}{1-q^{2k}}\psi(-q)\bigg)\\
	&=\bigg(\prod_{k=1}^{\infty}\frac{1}{1-q^{2k}}\bigg)\frac{1}{2}\bigg(\psi(q)-\psi(-q)\bigg)\\
	&=\bigg(\prod_{k=1}^{\infty}\frac{1}{1-q^{2k}}\bigg)\frac{1}{2}\bigg(\sum_{j=0}^{\infty}q^{\frac{j(j+1)}{2}}-\sum_{j=0}^{\infty}(-q)^{\frac{j(j+1)}{2}}\bigg).
	\end{align*}
By \eqref{Eq4} and \eqref{Eq5}, we have
\begin{align*}
\sum_{n=0}^{\infty}\frac{1-(-1)^n}{2}p_o(n)q^n & =\bigg(\prod_{k=1}^{\infty}\frac{1}{1-q^{2k}}\bigg)\frac{1}{2}\bigg(\sum_{i=0}^{\infty}2q^{8i^2-6i+1}+2q^{8i^2-10i+3}\bigg)\\
&= \sum_{k=0}^{\infty} p(k)q^{2k} \sum_{j=0}^{\infty}q^{t^o_j} =\sum_{k=0}^{\infty}\sum_{j=0}^{\infty}p(k)q^{2k+t^o_j}.
\end{align*}
Extracting the coefficient of $q^n$ in the identity	above, we obtain
\begin{align*}
\sum_{n=0}^{\infty}p_o(n)q^n=\sum_{n=0}^{\infty}\bigg(\sum_{j=0}^{\infty}p((n-t^o_j)/2)\bigg)q^n,
\end{align*}
from which \eqref{Eq3} follows.
\end{proof}

We recall that an overpartition of $n$, introduced in \cite{C-L}, is a partition  in which the first occurrence of a number may be overlined. For instance, there are eight overpartitions of 3:
$$3, \overline{3}, 2+1, \overline{2}+1, 2+\overline{1}, \overline{2}+\overline{1}, 1+1+1, \overline{1}+1+1.$$
In the next three results, we present recurrence relations for the number of overpartitions of $n$, denoted by $\overline{p}(n)$.

\begin{theorem}	For all $n \geq 0$, we have
\begin{align*}
&\overline{p}(n)-2\overline{p}(n-1)+2\overline{p}(n-4)-2\overline{p}(n-9)+2\overline{p}(n-16)-2\overline{p}(n-25)+\\
&2\overline{p}(n-36)-\cdots+2(-1)^j\overline{p}(n-j^2)+\cdots=
 \begin{cases} 1, & \text{if $n=0$,} \\ 0, & \text{otherwise.} \end{cases}
\end{align*}
\label{Teo3.3}
\end{theorem}

\begin{proof} We recall from \cite{C-L} that the generating function for overpartitions is given by
\begin{align*}
\sum_{n=0}^{\infty}\overline{p}(n)q^n=\prod_{k=1}^{\infty}\frac{(1+q^k)}{(1-q^k)}.
\end{align*}

We note that
\begin{align*}
\prod_{k=1}^{\infty}(1-q^k)&=\prod_{k=1}^{\infty}(1-q^k)\frac{(1-q^{2k-1})}{(1-q^{2k-1})}\frac{(1-q^{2k})}{(1-q^{2k})}\\
&=\prod_{k=1}^{\infty}\frac{(1-q^k)(1-q^{2k-1})(1-q^{2k})}{(1-q^{2k})(1-q^{2k-1})}\\
&=\prod_{k=1}^{\infty}\frac{(1-q^{2k-1})^2(1-q^{2k})}{(1-q^{2k-1})}\\
&=\varphi(-q)\prod_{k=1}^{\infty}\frac{1}{(1-q^{2k-1})}.
\end{align*}
By Euler's identity, we have
\begin{align*}
\prod_{k=1}^{\infty}(1-q^k)=\varphi(-q)\prod_{k=1}^{\infty}(1+q^k),
\end{align*}
and, then,
\begin{align*}
\prod_{k=1}^{\infty}\frac{1}{(1-q^k)}=\frac{1}{\varphi(-q)}\prod_{k=1}^{\infty}\frac{1}{(1+q^k)}.
\end{align*}
Hence we obtain the following equivalent equalities
\begin{align*}
\sum_{n=0}^{\infty}\overline{p}(n)q^n&=\frac{\prod_{k=1}^{\infty}(1+q^k)}{\varphi(-q)\prod_{k=1}^{\infty}(1+q^k)}\\
\varphi(-q)\sum_{n=0}^{\infty}\overline{p}(n)q^n&=1\\
\bigg(1+2\sum_{k=1}^{\infty}(-1)^kq^{k^2}\bigg)\sum_{n=0}^{\infty}\overline{p}(n)q^n&=1\\
\sum_{n=0}^{\infty}\bigg(\overline{p}(n)q^n+2\sum_{k=1}^{\infty}(-1)^k\overline{p}(n)q^{k^2+n}\bigg)&=1 \\
\sum_{n=0}^{\infty}\bigg(\overline{p}(n)+2\sum_{k=1}^{\infty}(-1)^k\overline{p}(n-k^2)\bigg)q^n&=1.
\end{align*}
The result now follows from the last equality.
\end{proof}

Our second recurrence for $\overline{p}(n)$ involves $p_d(n)$.

\begin{theorem}	For all $n \geq 1$, we have
\begin{align*}
\overline{p}(n)&-\overline{p}(n-1)-\overline{p}(n-2)+\overline{p}(n-5)+\overline{p}(n-7)-\cdots\\
&\cdots+(-1)^j\left( \overline{p}(n-j(3j-1)/2)+\overline{p}(n-j(3j+1)/2)\right)+\cdots= p_d(n).
\end{align*}
\label{Teo3.4}
\end{theorem}

\begin{proof} We have
	\begin{align*}
	\prod_{k=1}^{\infty}(1+q^k)&=\prod_{k=1}^{\infty}\frac{(1+q^k)}{(1-q^k)}\prod_{k=1}^{\infty}(1-q^k)\\
	&= \sum_{k=1}^{\infty}\overline{p}(k)q^k \sum_{k=-\infty}^{\infty}(-1)^jq^{\frac{j(3j-1)}{2}} \\
	&= \sum_{k=0}^{\infty}\overline{p}(k)q^k \bigg(1+\sum_{j=1}^{\infty}(-1)^jq^{\frac{j(3j-1)}{2}}+\sum_{j=1}^{\infty}(-1)^jq^{\frac{j(3j+1)}{2}}\bigg)\\
	&=\sum_{k=0}^{\infty}\bigg(\overline{p}(k)q^k+\sum_{j=1}^{\infty}(-1)^j\overline{p}(k)q^{k+\frac{j(3j-1)}{2}}+\sum_{j=1}^{\infty}(-1)^j\overline{p}(k)q^{k+\frac{j(3j+1)}{2}}\bigg).
	\end{align*}
Therefore,
\begin{align*}
\sum_{n=0}^{\infty}&p_d(n)q^n =\sum_{n=0}^{\infty}\bigg(\overline{p}(n)+\sum_{j=1}^{\infty}(-1)^j\left( \overline{p}(n-j(3j-1)/2)+\overline{p}(n-j(3j+1)/2)\right) \bigg)q^n,
\end{align*}
from which the proof follows by comparing coefficients of $q^n$ on both sides of the last equation.
\end{proof}

If $\overline{p_d}(n)$ denotes the number of overpartitions of $n$ into distinct parts, then we have the following recurrence for $\overline{p}(n)$.
	
	\begin{theorem} For all $n \geq 0$, we have
		\begin{align*}
		\overline{p}(n)-& \overline{p}(n-2)-\overline{p}(n-4)+\overline{p}(n-10)+\overline{p}(n-14)-\cdots\\
		\cdots&+(-1)^j\left( \overline{p}(n-j(3j-1))+\overline{p}(n-j(3j+1)) \right) +\cdots = \overline{p_d}(n).
		\end{align*}
	\end{theorem}
	
	\begin{proof} By \eqref{Euler} we have
		\begin{align*}
		\sum_{k=0}^{\infty} \overline{p_d}(n) q^n&=\prod_{k=1}^{\infty}(1+q^{k})^2 =\prod_{k=1}^{\infty}\frac{(1-q^{2k})(1+q^k)}{1-q^{k}} \\
		&=(q^2;q^2)_\infty \sum_{k=0}^{\infty} \overline{p}(k)q^k =\sum_{k=0}^{\infty} \overline{p}(k)q^k \sum_{j=-\infty}^{\infty} (-1)^jq^{j(3j-1)}  \\
		&= \sum_{k=0}^{\infty} \overline{p}(k)q^k \bigg(1+\sum_{j=1}^{\infty}(-1)^jq^{j(3j-1)}+\sum_{j=1}^{\infty}(-1)^jq^{j(3j+1)}\bigg)\\
		&=\sum_{n=0}^{\infty}\bigg(\overline{p}(n)+\sum_{j=1}^{\infty}(-1)^j\left( \overline{p}(n-j(3j-1))+\overline{p}(n-j(3j+1))\right) \bigg)q^n.
		\end{align*}	
		Thus, the result follows from extracting the coefficient of $q^n$ on both sides of the last equality.
\end{proof}

Now we prove a recurrence relations satisfied by $qq(n)$, the number of partitions of $n$ into distinct odd parts.

\begin{theorem} For all $n \geq 1$, we have
\begin{align*}
qq(n)-qq(n-4)-qq(n-8)&+qq(n-20)+qq(n-28)-qq(n-48)-\\
qq(n-60)+\cdots+(-1)^j & \left( qq(n-2j(3j-1))+qq(n-2j(3j+1)) \right) \\
&\cdots= \begin{cases} 1, & \text{if $n$ is a triangular number,} \\ 0, & \text{otherwise.} \end{cases}
\end{align*}
\label{Teo3.5}
\end{theorem}

\begin{proof} It is easy to see that
\begin{align*}
\prod_{k=1}^{\infty}(1+q^{2k-1})&=\prod_{k=1}^{\infty}\frac{(1+q^k)}{(1+q^{2k})} =\prod_{k=1}^{\infty}\frac{(1+q^k)(1-q^{2k})}{(1-q^{4k})}.
\end{align*}
Thus,
\begin{align*}
\sum_{j=0}^{\infty}qq(j)q^j=\prod_{k=1}^{\infty}\frac{(1+q^k)(1-q^{2k})}{(1-q^{4k})},
\end{align*}
which can be rewritten as
\begin{align*}
\prod_{k=1}^{\infty}(1-q^{4k})\sum_{j=0}^{\infty}qq(j)q^j&=\sum_{i=0}^{\infty}p_d(i)q^i\prod_{k\geq1}(1-q^{2k}).
\end{align*}
That is to say
\begin{align*}
(q^4;q^4)_\infty\sum_{j=0}^{\infty}qq(j)q^j&=(q^2;q^2)_\infty \sum_{i=0}^{\infty}p_d(i)q^i.
\end{align*}
Then, by \eqref{Euler}, we have
\begin{align*}
(q^4;q^4)_\infty&\sum_{j=0}^{\infty} qq(j)q^j\\
&=\bigg(1+\sum_{k=1}^{\infty}(-1)^kq^{2j(3j-1)}+\sum_{k=1}^{\infty}(-1)^kq^{2j(3j+1)}\bigg)\sum_{j=0}^{\infty} qq(j)q^j\\
&=\sum_{n=0}^{\infty}\bigg(qq(n)+\sum_{k=0}^{\infty}(-1)^k\left( qq(n-2j(3j-1))+qq(n-2j(3j+1))\right) \bigg)q^n.
\end{align*}
On the other hand, we have
\begin{align*}
(q^2;q^2)_\infty & \sum_{i=0}^{\infty}p_d(i)q^i \\
&=\sum_{i=0}^{\infty}p_d(i)q^i \bigg(1+\sum_{k=1}^{\infty}(-1)^kq^{k(3k-1)}+\sum_{k=1}^{\infty}(-1)^kq^{k(3k+1)}\bigg)\\
&=\sum_{n=0}^{\infty}\bigg(p_d(n)+\sum_{k=1}^{\infty}(-1)^k\left( p_d(n-k(3k-1))+p_d(n-k(3k+1))\right) \bigg)q^n.
\end{align*}
Hence,
\begin{align*}
&\sum_{n=0}^{\infty}\bigg(qq(n)+\sum_{k=1}^{\infty}(-1)^k\left( qq(n-2j(3j-1))+qq(n-2j(3j+1))\right) \bigg)q^n\\
&=\sum_{n=0}^{\infty}\bigg(p_d(n)+\sum_{k=1}^{\infty}(-1)^k\left( p_d(n-k(3k-1))+p_d(n-k(3k+1))\right) \bigg)q^n.
\end{align*}
The result follows from extracting the coefficient of $q^n$ on both sides of the last equation and using Theorem 1 of \cite{Ono}.
\end{proof}

The next theorem presents a recurrence for the number of partitions into odd parts, denoted by $p_o(n)$.

\begin{theorem} For all $n \geq 0$, we have
\begin{align*}
p_o(n)-&p_o(n-1)-p_o(n-5)+p_o(n-8)+p_o(n-16)-\cdots\\
\cdots+&(-1)^j\big[p_o(n-j(3j-2))+p_o(n-j(3j+2))\big]+\cdots\\
\cdots&= \begin{cases} 1, & \text{if $n$ is  3 times a triangular number,} \\ 0, & \text{otherwise.} \end{cases}
\end{align*}
\label{Teo3.6}
\end{theorem}

\begin{proof} Setting $a=-q$ and $b=-q^5$ in \eqref{eq1} and \eqref{eq5} we obtain
\begin{align*}
(q;q^6)_\infty(q^5;q^6)_\infty(q^6;q^6)_\infty =	f(-q,-q^5)=\sum_{j=-\infty}^{\infty} (-q)^{\frac{j(j+1)}{2}}(-q^5)^{\frac{j(j-1)}{2}},
\end{align*}
from which it follows that
\begin{align}
\sum_{j=-\infty}^{\infty} (-1)^jq^{j(3j-2)}&=\prod_{k=1}^{\infty}(1-q^{6k-5})(1-q^{6k-1})(1-q^{6k}) \label{eq35} \\      &=\prod_{k=1}^{\infty}(1-q^{6k-5})(1-q^{6k-1})(1-q^{6k})\frac{(1-q^{6k-3})}{(1-q^{6k-3})} \nonumber\\
&=\prod_{k=1}^{\infty}\frac{(1-q^{2k-1})(1-q^{6k})}{(1-q^{6k-3})}.\nonumber
\end{align}
Then, by \eqref{eq3}, we have
\begin{align*}
\psi(q^3) & = \prod_{k=1}^{\infty}\frac{1}{(1-q^{2k-1})}\sum_{j=-\infty}^{\infty}(-1)^jq^{j(3j-2)}\\
& =  \sum_{i=0}^{\infty}p_o(i)q^i \bigg(1+\sum_{j=1}^{\infty}(-1)^jq^{j(3j-2)}+\sum_{j=1}^{\infty}(-1)^jq^{j(3j+2)}\bigg) \\
& = \sum_{n=0}^{\infty}\bigg(p_o(n)+\sum_{j=1}^{\infty}(-1)^j\left( p_o(n-j(3j-2))+p_o(n-j(3j+2))\right) \bigg)q^n.
\end{align*}
The result follows by comparing the coefficients of $q^n$ on both sides of the last expression.
\end{proof}

Let $\ell$ be a positive integer. A partition of $n$ having no part divisible by $\ell$ is called an $\ell$-regular partition of $n$. Let $b_{\ell}(n)$ denote the number of $\ell$-regular partition of $n$. The generating function of $b_{\ell}(n)$ is 
\begin{equation*}
\sum_{n=0}^{\infty} b_{\ell}(n)q^n = \frac{(q^{\ell};q^{\ell})_{\infty}}{(q;q)_{\infty}}.
\end{equation*}
Our next result is a recurrence relation for $p(n)$ involving $b_{\ell}(n)$.

\begin{theorem} Let $\ell \geq 1$. For all $n \geq 0$, we have
	\begin{align*}
	&p(n)-p(n-\ell)-p(n-2\ell)+p(n-5\ell)+p(n-7\ell)-\cdots\\
	&\cdots+(-1)^j\big[p(n-\ell j(3j-1)/2)+p(n-\ell j(3j+1)/2)\big]+\cdots=b_{\ell}(n).
	\end{align*}
	\label{Teo3.7}
\end{theorem}

\begin{proof} We have
\begin{align*}
&\sum_{n=0}^{\infty} b_{\ell}(n)q^n\\
&=\prod_{k=1}^{\infty}\frac{(1-q^{\ell k})}{(1-q^k)}\\
&=\sum_{n=0}^{\infty}p(n)q^n\bigg(\sum_{j=-\infty}^{\infty}(-1)^jq^{\ell\frac{j(3j-1)}{2}}\bigg)\\
&=\sum_{n=0}^{\infty}p(n)q^n\bigg(1+\sum_{j=1}^{\infty}(-1)^jq^{\ell\frac{j(3j-1)}{2}}+\sum_{j=1}^{\infty}(-1)^jq^{\ell\frac{j(3j+1)}{2}}\bigg)\\
&=\sum_{n=0}^{\infty}\bigg(p(n)+\sum_{j=1}^{\infty}(-1)^j\big[p(n)(n-\ell j(3j-1)/2)+p(n-\ell j(3j+1)/2))\big]\bigg)q^n,
\end{align*}
from which the result follows.
\end{proof}

We close this section with a recurrence relation for the number of partitions of $n$ having parts congruent to $\pm c \pmod{m}$.

\begin{theorem}	Given integers $a$ and $m\geq1$, we let $p^c_m(n)$ denote the number of partitions of $n$ having parts congruent to $\pm c$ modulo $m$. Then, for all $n \geq 0$,
	\begin{align*}
	p^c_m(n)-&p^c_m(n-(m-c))-p^c_m(n-c)+p^c_m(n-(3m-2c))\\
	+p^c_m&(n-(m+2c))+\cdots+(-1)^jp^c_m(n-(mj^2+(m-2c)j)/2)\\
	&+(-1)^jp^c_m(n-(mj^2-(m-2c)j)/2)+\cdots =  \begin{cases} 1, & \text{if $n=mk_j^e$,} \\ -1, & \text{if $n=mk_j^o$,} \\ 0, & \text{otherwise,} \end{cases}
	\end{align*}
where $k_j^e$ (resp. $k_j^o$)  is the $j$-th even (resp. odd) pentagonal number.
	\label{Teo3.8}
\end{theorem}

\begin{proof} Setting $c=-q^{m-c}$ and $b=-q^c$ in \eqref{eq1} and \eqref{eq5}, we obtain
\begin{align*}
(q^{m-c};q^m)_\infty(q^c;q^m)_\infty(q^m;q^m)_\infty & =	f(-q^{m-c},-q^c) \\ & =\sum_{j=-\infty}^{\infty}(-q^{m-c})^{\frac{j(j+1)}{2}}(-q^c)^{\frac{j(j-1)}{2}},
	\end{align*}
which yields
\begin{align}
\sum_{j=-\infty}^{\infty}(-1)^jq^{\frac{mj^2+(m-2c)j}{2}}&=\prod_{k=1}^\infty (1-q^{mk-c})(1-q^{mk-(m-c)})(1-q^{mk}).
\label{eq36}
\end{align}
	
The generating function for $p^c_m$ is given by
\begin{align*}
\sum_{i=0}^\infty p^c_m(i)q^i=\prod_{k=1}^\infty\frac{1}{(1-q^{mk-c})(1-q^{mk-(m-c)})}.
\end{align*}
Hence, we can rewrite \eqref{eq36} as
\begin{align*}
\prod_{k=1}^\infty \frac{1}{(1-q^{mk-c})(1-q^{mk-(m-c)})} \sum_{j=-\infty}^\infty (-1)^jq^{\frac{mj^2+(m-2c)j}{2}} &=\prod_{k=1}^\infty(1-q^{mk}),
\end{align*}
or, equivalently,
\begin{align*}
\sum_{i=0}^\infty p^c_m(i)q^i \sum_{j=-\infty}^\infty (-1)^jq^{\frac{mj^2+(m-2c)j}{2}} &=\sum_{j=-\infty}^\infty(-1)^jq^{\frac{mj(3j-1)}{2}}.
\end{align*}
This last identity yields
\begin{align*}
\sum_{i=0}^\infty p^c_m(i)q^i\bigg(1+\sum_{j=1}^\infty (-1)^jq^{\frac{mj^2+(m-2c)j}{2}}+&\sum_{j=1}^\infty (-1)^jq^{\frac{mj^2-(m-2c)j}{2}}\bigg)\\
=\sum_{n=0}^\infty \bigg(p^c_m(n)+\sum_{j=1}^\infty(-1)^jp^c_m(n-(mj^2&+(m-2c)j)/2)\\
+(-1)^j&p^c_m(n-(mj^2-(m-2c)j)/2)\bigg)q^n.
\end{align*}
Therefore,
\begin{align*}
\sum_{j=-\infty}^{\infty} (-1)^jq^{\frac{mj(3j-1)}{2}} = \sum_{n=0}^{\infty}\bigg(p^c_m(n)+\sum_{j\geq1}(-1)^j&p^c_m(n-(mj^2+(m-2c)j)/2)\\
+(-1)^j&p^c_m(n-(mj^2-(m-2c)j)/2)\bigg)q^n,
\end{align*}
which completes the proof. 
\end{proof}

As special cases of the above theorem, we have the following corollaries which provide recurrence relations for the number of partitions that appear in the well-known Rogers-Ramanujan's identities.

\begin{corollary} Let $p_{R1}(n)$ denote the number of partitions of $n$ whose parts are congruent to $\pm 1$ modulo $5$ and let $p_{R2}(n)$ denote the number of partitions of $n$ whose parts are congruent to $\pm 2$ modulo $5$. Then, for all $n \geq 0$, 
\begin{align*}
p_{R1}(n)&-p_{R1}(n-1)-p_{R1}(n-4)+p_{R1}(n-7)+p_{R1}(n-13)-\cdots \\
\cdots&+(-1)^j\big[p_{R1}(n-j(5j-3)/2)+p_{R1}(n-j(5j+3)/2)\big]+\cdots\\
&\cdots= \begin{cases} 1, & \text{if $n=5k_j^e$} \\ -1, & \text{if $n=5k_j^o$} \\ 0, & \text{otherwise,} \end{cases}
\end{align*}
and
\begin{align*}
p_{R2}&(n)-p_{R2}(n-1)-p_{R2}(n-2)+p_{R2}(n-5)+p_{R2}(n-7)-\cdots \\
\cdots&+(-1)^j\big[p_{R2}(n-j(3j-1)/2)+p_{R2}(n-j(3j+1)/2)\big]+\cdots\\
&\cdots=  \begin{cases} 1, & \text{if $n=h_j^e$} \\ -1, & \text{if $n=h_i^o$} \\ 0, & \text{otherwise,} \end{cases}
\end{align*}
where $h_j^e$ (resp. $h_j^o$) is the $j$-th heptagonal number with $j$ even (resp. odd).
\label{Corol3.1}
\end{corollary}

\begin{corollary} Let $s_1(n)$ denote the number of partitions of $n$ having congruent to $\pm 1$ modulo $6$ and let $s_2(n)$ denote the number of partitions of $n$ whose parts are congruent to $\pm 2$ modulo $6$. Then, for all $n \geq 0$,
\begin{align*}
s_1(n)&-s_1(n-1)-s_1(n-5)+s_1(n-8)+s_1(n-16)-\cdots \\
\cdots&+(-1)^j\big[s_1(n-j(3j-2))+s_1(n-j(3j+2))\big]+\cdots\\
=s_2&(n)-s_2(n-2)-s_2(n-4)+s_2(n-10)+s_2(n-14)-\cdots \\
\cdots&+(-1)^j\big[s_2(n-j(3j-1))+s_2(n-j(3j+1))\big]+\cdots\\
&\cdots=  \begin{cases} 1, & \text{if $n=6k_j^e$} \\ -1, & \text{if $n=6k_i^o$} \\ 0, & \text{otherwise.} \end{cases}
\end{align*}
\label{Corol3.2}
\end{corollary}


\section{Acknowledgments}

This research was supported by S\~ao Paulo Research Foundation (FAPESP) (grant no. 2019/10742-0).


\begin{thebibliography}{99}
\bibitem{Berndt} B. C. Berndt, Number Theory in the Spirit of Ramanujan, American Mathematical Society, 2006.
\bibitem{Sills} Y. Choliy, L. W. Kolitsch, and A. V. Sills, {\it Partition recurrences}, Integers {\bf 18} (2018), Article A1.
\bibitem{C-L} S. Corteel and J. Lovejoy, {\it Overpartitions}, Trans. Amer. Math. Soc. {\bf 356} (2004), 1623--1635.
\bibitem{Ewell} J. A. Ewell, {\it Partition recurrences}, J. Combin. Theory Ser. A {\bf 14} (1973), 125--127.
\bibitem{Ewell1} J. A. Ewell, {\it Another recurrence for the partition function}, JP J. Algebra, Number Theory, App. {\bf 4} (2004), 147--152.
\bibitem{Merca} M. Merca, {\it New recurrences for Euler's partition function}, Turk. J. Math. {\bf 41} (2017), 1184--1190.
\bibitem{N} D. Nyirenda, {\it On parity and recurrences for certain partition functions}, Contrib. Discrete Math. {\bf 15} (2020), 72--79.
\bibitem{Ono} K. Ono, N. Robbins, and B. Wilson, {\it Some recurrences for arithmetical functions}, J. Indian Math. Soc.  {\bf 62} (1996), pages 29--50.
\end{thebibliography}
\end{document}